\documentclass[12pt]{amsart}
\title{The \'etale open topology over the fraction\\ field of a henselian local domain}
\author{Will Johnson, Erik Walsberg, Jinhe Ye }
\email{willjohnson@fudan.edu.cn, ewalsber@uci.edu, jinhe.ye@imj-prg.fr}

\usepackage{amsmath, amssymb, amsthm, enumitem, comment}

\usepackage[normalem]{ulem}
\usepackage{tikz-cd}
\usepackage[Symbol]{upgreek}
\usepackage[mathscr]{euscript}
\usepackage{fullpage} 	
\usepackage{hyperref}
\usepackage{lineno}
\usepackage[all]{xy}
\usepackage{centernot}
\usepackage{xcolor}

\DeclareFontFamily{U}{BOONDOX-calo}{\skewchar\font=45 }
\DeclareFontShape{U}{BOONDOX-calo}{m}{n}{
  <-> s*[1.05] BOONDOX-r-calo}{}
\DeclareFontShape{U}{BOONDOX-calo}{b}{n}{
  <-> s*[1.05] BOONDOX-b-calo}{}
\DeclareMathAlphabet{\mathcalboondox}{U}{BOONDOX-calo}{m}{n}
\SetMathAlphabet{\mathcalboondox}{bold}{U}{BOONDOX-calo}{b}{n}
\DeclareMathAlphabet{\mathbcalboondox}{U}{BOONDOX-calo}{b}{n}

\DeclareMathOperator*{\forkindep}{\raise0.2ex\hbox{\ooalign{\hidewidth$\vert$\hidewidth\cr\raise-0.9ex\hbox{$\smile$}}}}

\DeclareFontFamily{U}{fsy}{}
\DeclareFontShape{U}{fsy}{m}{n}{<->s*[.9]psyr}{}
\DeclareSymbolFont{der@m}{U}{fsy}{m}{n}
\DeclareMathSymbol{\der}{\mathord}{der@m}{182}

\newcommand{\mfrak}{\mathfrak{m}}
\newcommand{\pfrak}{\mathfrak{p}}

\newcommand{\Sa}[1]{\ensuremath{\mathscr{#1}}}

\newcommand{\Spec}{\operatorname{Spec}}

\newcommand{\Frac}{\operatorname{Frac}}

\newcommand{\res}{\operatorname{res}}

\newcommand{\Chara}{\operatorname{Char}}

\newtheorem*{claim-star}{Claim}
\newtheorem{theorem}{Theorem}[section] 
\newtheorem{lemma}[theorem]{Lemma}

\newtheorem{prop-def}[theorem]{Proposition-Definition}
\newtheorem{corollary}[theorem]{Corollary}
\newtheorem{fact}[theorem]{Fact}
\newtheorem{fact-eh}[theorem]{Fact(?)}

\newtheorem{proposition}[theorem]{Proposition}
\newtheorem{proposition-eh}[theorem]{Proposition(?)}
\newtheorem*{theorem-star}{Theorem}
\newtheorem*{conjecture-star}{Conjecture}
\newtheorem*{question-star}{Question}
\newtheorem*{lemma-star}{Lemma}

\theoremstyle{definition}

\newtheorem{remark}[theorem]{Remark}
\theoremstyle{remark}

\newcommand{\Aa}{\mathbb{A}}

\newcommand{\Gg}{\mathbb{G}}
\newcommand{\Qq}{\mathbb{Q}}
\newcommand{\Rr}{\mathbb{R}}
\newcommand{\Zz}{\mathbb{Z}}
\newcommand{\Nn}{\mathbb{N}}
\newcommand{\Cc}{\mathbb{C}}

\begin{document}

\maketitle

\begin{abstract}
Suppose that $R$ is a local domain with fraction field $K$.
If $R$ is Henselian then the $R$-adic topology over $K$ refines the \'etale open topology.
If $R$ is regular then the \'etale open topology over $K$ refines the $R$-adic topology.
In particular the \'etale open topology over $L((t_1,\ldots,t_n))$ agrees with the $L[[t_1,\ldots,t_n]]$-adic topology for any field $L$ and $n \ge 1$.
\end{abstract}

\section{Introduction}
Throughout, $K$ and $L$ are fields, all rings are commutative with unit, the ``dimension'' of a ring is the Krull dimension, and by convention a ``local ring'' is not a field.
The \'etale open topology (or $\Sa E_K$-topology) on a $K$-variety $V$ is a topology on the set $V(K)$ of $K$-points introduced in \cite{firstpaper}; see Section~\ref{section:conventions and background} for definitions.  The $\Sa E_K$-topology recovers the ``natural'' topology on $V(K)$ in several cases of interest.  For example, when $K$ is algebraically closed, the \'etale open topology is the Zariski topology on $V(K)$, and when $K$ is $\Rr$ or $\Qq_p$ (or any local field other than $\Cc$), the \'etale open topology is the analytic topology on $V(K)$.  The \'etale open topology is primarily of interest when $K$ is a large field in the sense of Pop \cite{pop-embedding}.  In fact, the \'etale open topology is discrete when $K$ is non-large~\cite[Theorem C]{firstpaper}.  (See \cite{Pop-little, open-problems-ample} for more information on largeness, including a definition.)  The \'etale open topology has applications to large fields: it yielded a classification of model-theoretically stable large fields, and provided new insights into several known results on large fields \cite[Theorem D, 9.2]{firstpaper}.

 If $\Sa T$ is a field topology on $K$, then $\Sa T$ induces a topology on $V(K)$ in a natural way.\footnote{When $V$ is $\Aa^n$ we take the product topology on $K^n$.
When $V \subseteq \Aa^n$ is an affine variety we take the subspace topology on $V(K) \subseteq K^n$.
When $V$ is an arbitrary variety, we cover $V$ by affine open subvarieties and glue the corresponding topologies.
See \cite[Chapter I.10]{red-book} for details.}
We call the resulting topology the $\Sa T$-topology on $V$.  In cases like $K = \Rr$ or $K = \Qq_p$, the \'etale open topology is induced by a field topology.  On the other hand, this fails when $K$ is algebraically closed.  (The Zariski topology on $V \times W$ is not the product topology of the Zariski topologies on $V$ and $W$.)  One might hope to characterize when the \'etale open topology is induced by a field topology.  Some results in this direction were obtained in \cite[Theorem B and Section 8]{firstpaper}.  In the current paper, we give a new example where the \'etale open topology is induced by a field topology, namely, when $K$ is the fraction field of a Henselian regular local ring such as $\Cc[[x,y]]$.

Fact~\ref{fact:canonical} describes what we have established so far with Tran concerning the relationship between the \'etale open and other topologies, in order~\cite[Theorem A, Proposition 6.1, Proposition 6.14, Theorem B, Theorem 6.15, Theorem B]{firstpaper}.

\begin{fact}
\label{fact:canonical}
Suppose that $V$ is a $K$-variety and $v$ is a non-trivial valuation on $K$.
\begin{enumerate}
\item The $\Sa E_K$-topology on $V(K)$ refines the Zariski topology.
\item If $K$ is separably closed then the $\Sa E_K$-topology on $V(K)$ agrees with the Zariski topology.
\item If $<$ is a field order on $K$ then the $\Sa E_K$-topology on $V(K)$ refines the $<$-topology.
\item If $K$ is real closed then the $\Sa E_K$-topology on $V(K)$ agrees with the order topology.
\item  If the Henselization of $(K,v)$ is not separably closed then the $\Sa E_K$-topology on $V(K)$ refines the $v$-topology.
(Hence if the value group of $v$ is not divisible or the residue field of $v$ is not algebraically closed then the $\Sa E_K$-topology on $V(K)$ refines the $v$-topology.)
\item If $K$ is not separably closed and $v$ is Henselian then the $\Sa E_K$-topology on $V(K)$ agrees with the $v$-topology.
\end{enumerate}
\end{fact}

If $R$ is a Henselian valuation ring with fraction field $K$, then Fact~\ref{fact:canonical}.6 says that the \'etale open topology over $K$ is the valuation topology, unless $K$ is separably closed or the valuation is trivial.  For example, the \'etale open topology over $L((t))$ is the valuation topology.  It is natural to ask whether Fact~\ref{fact:canonical}.6 generalizes to fraction fields of Henselian local domains such as $L[[t_1,\ldots,t_n]]$.  We first need an analogue of the valuation topology.  If $R$ is a local domain with fraction field $K$, then
$\{\alpha R + \beta : \alpha \in K^\times, \beta \in K\}$ is a basis of opens for a non-discrete Hausdorff field topology on $K$ \cite[Theorem~2.2]{Prestel1978}.
We call this the \textbf{$R$-adic topology} on $K$.
This is the coarsest ring topology on $K$ with $R$ open.
When $R$ is a non-trivial valuation ring, the $R$-adic topology is the valuation topology.
As mentioned above, the $R$-adic topology induces a topology on $V(K)$ for each $K$-variety $V$.

\begin{theorem}
\label{thm:main}
Suppose that $R$ is a local domain with fraction field $K$ and $V$ is a $K$-variety.
\begin{enumerate}
\item If $R$ is Henselian then the $R$-adic topology on $V(K)$ refines the $\Sa E_K$-topology.
\item If $R$ is regular then the $\Sa E_K$-topology on $V(K)$ refines the $R$-adic topology.
    \end{enumerate}
Hence the \'etale open topology over the fraction field $L((t_1,\ldots,t_n))$ of $L[[t_1,\ldots,t_n]]$ agrees with the $L[[t_1,\ldots,t_n]]$-adic topology.
\end{theorem}

See Section~\ref{section:conventions and background} for the definitions of regularity and Henselianity.
Examples of  regular Henselian local rings are the ring $L[[t_1,\ldots,t_n]]$ of formal power series for any field $L$, and the ring $L\{t_1,\ldots,t_n\}$ of convergent power series for $L$ a local field.  Additionally, the ring $L(t_1,\ldots,t_n)^{\mathrm{alg}} \cap L[[t_1,\ldots,t_n]]$ is a regular Henselian local ring for any field $L$, as it is the Henselization of the regular local ring $L[t_1,\ldots,t_n]_{(t_1,\ldots,t_n)}$ (see \cite[07PX, 07PV, 0A1W, 06LN]{stacks-project}).

A one-dimensional regular local ring is a discrete valuation ring~\cite[11.1]{Eisenbud}, so the one-dimensional case of Theorem~\ref{thm:main}.2 follows from Fact~\ref{fact:canonical}.5.
If $R$ is a Noetherian local domain of dimension at least two then the $R$-adic topology is not induced by a valuation (see Proposition~\ref{prop:V-topology}) so Theorem~\ref{thm:main} does not follow from Fact~\ref{fact:canonical}.

As noted above, the \'etale open topology is connected to the class of large fields in field theory.
Specifically, the \'etale open topology on $K = \Aa^1(K)$ is non-discrete if and only if $K$ is large~\cite[Theorem C]{firstpaper}.
Thus the first claim of Theorem~\ref{thm:main} can be seen as a topological refinement of the fact, proven in \cite[Theorem 1.1]{Pop-henselian}, that fraction fields of Henselian local domains are large.

Both refinements in Theorem~\ref{thm:main} can be strict.
For example the localization of $L[t_1,\ldots,t_n]$ at the ideal generated by $t_1,\ldots,t_n$ is a regular local ring $R$ with non-large fraction field $L(t_1,\ldots,t_n)$, so the \'etale open topology over $L(t_1,\ldots,t_n)$ is discrete and hence strictly refines the $R$-adic topology.
Theorem~\ref{thm:example} shows that the other refinement may also be strict.

\begin{theorem}
\label{thm:example}
Fix a prime $p$.
There is a subring $E$ of $\Zz_p$ such that:
\begin{enumerate}
\item $E$ is a one-dimensional Noetherian Henselian local domain,
\item $\Qq_p$ is the fraction field of $E$, and
\item the $E$-adic topology on $\Qq_p$ strictly refines the $p$-adic topology.
\end{enumerate}
\end{theorem}

By Fact~\ref{fact:canonical}.6 the \'etale open topology over $\Qq_p$ agrees with the $p$-adic topology, hence the $E$-adic topology on $\Qq_p$ strictly refines the  \'etale open topology.
Therefore some assumption like regularity is needed in Theorem~\ref{thm:main}.2. In an upcoming paper with Dittmann, we will show that excellence is another sufficient condition.  In particular, we will show that for an excellent Henselian local domain $R$, the $R$-adic topology on the fraction field is the \'etale open topology.  (In the present paper, the ring $E$ of Theorem~\ref{thm:example} has non-reduced completion and is hence not excellent; see Remark~\ref{remark} below.)  More generally, the upcoming paper suggests that the comparison between the $R$-adic topology and \'etale open topology is connected to resolution of singularities over $R$.

We now discuss an application to definable sets.
``Definable'' means ``first-order definable in the language of rings, possibly with parameters'' and a field is \textit{Henselian} if it admits a non-trivial Henselian valuation.
It is a well-known and important fact that definable sets in characteristic zero Henselian fields are well behaved with respect to the valuation.
In particular if $K$ is a characteristic zero Henselian valued field then every definable subset of $K^n$ is a finite union of valuation open subsets of Zariski closed subsets of $K^n$~\cite{lou-dimension}; note that this applies to $L((t))$.
Definable sets in fraction fields of characteristic zero Henselian local domains need not be well behaved. The following roughly follows the argument in~\cite[Example 10]{Fehm-subfield} for $\mathbb{Q}((X,Y))$.
Suppose $\Chara(L) = 0$ and $n \ge 2$.
Jensen and Lenzig~\cite[Theorem~3.34]{model-theoretic-algebra} showed that $L((t_1,\ldots,t_n))$ defines $L[[t_1,\ldots,t_n]]$,
Becker and Lipschitz~\cite{becker-lipschitz} showed that $L[[t_1,\ldots,t_n]]$ defines $\Nn$, Delon~\cite{delon-multivariable} showed that $L[[t_1,\ldots,t_n]]$ uniformly defines all subsets of $\Nn$, hence $L((t_1,\ldots,t_n))$ defines the standard model of second order arithmetic.
Nevertheless, the existentially definable sets are still somewhat tame, as observed in~\cite{ansco-exis,Fehm-subfield,secondpaper}.

\begin{corollary}
\label{cor:exis-def}
Suppose that $R$ is a Henselian local domain  with fraction field $K$, $K$ is perfect, and $X \subseteq K^n$ is existentially definable.
Then there are pairwise disjoint irreducible smooth subvarieties $V_1,\ldots,V_k$ of $\Aa^n$ and $O_1,\ldots,O_k$ such that each $O_i$ is a definable $R$-adically open subset of $V_i(K)$ and $X = O_1 \cup \ldots \cup O_k$.
\end{corollary}

Corollary~\ref{cor:exis-def} follows directly from Theorem~\ref{thm:main}.1, and the theorem, proven in~\cite[Theorem B]{secondpaper}, that if $K$ is perfect then every existentially definable subset of $K^n$ is a finite union of definable $\Sa E_K$-open subsets of Zariski closed subsets of $K^n$.

\subsection*{Acknowledgments}
We would like to thank Marcus Tressl who encouraged us to write this down, as well as two anonymous referees, who caught many errors. WJ was partially supported by the National Natural Science Foundation of China (Grant No.\@ 12101131). JY was partially supported by GeoMod AAPG2019 (ANR-DFG), Geometric and
Combinatorial Configurations in Model Theory.

\subsection{Conventions and background}
\label{section:conventions and background}
A $K$-variety is a separated  $K$-scheme of finite type, $n$ is a natural number, $\Aa^n$ is $n$-dimensional affine space over $K$, and $\Gg_m$ is the scheme-theoretic multiplicative group over $K$, i.e. $\Aa^n = \Spec K[x_1,\ldots,x_n]$ and $\Gg_m = \Spec K[y,y^{-1}]$.
We let $V(K)$ be the set of $K$-points of a $K$-variety $V$.
Recall that $\Aa^n(K) = K^n$ and $\Gg_m(K) = K^\times$.
We let $\Frac(R)$ be the fraction field of a domain $R$.

Suppose that $R$ is a local ring with maximal ideal $\mfrak$.
Then $R$ is \textbf{Henselian} if for any $g \in R[x]$ and $\alpha \in R$ such that $g(\alpha) \equiv 0 \pmod{\mfrak}$ and $g'(\alpha) \not\equiv 0 \pmod{\mfrak}$ there is $\alpha^* \in R$ such that $g(\alpha^*) = 0$ and $\alpha^* \equiv \alpha \pmod{\mfrak}$.
Likewise, $R$ is \textbf{regular} if $R$ is Noetherian and $\mfrak$ admits a $d$-element generating set, where $d = \dim R$.

We briefly recall the \'etale open topology.
Suppose that $V$ is a $K$-variety.
An \textbf{\'etale image} in $V(K)$ is a set of the form $f(X(K))$ for an \'etale morphism $f \colon X \to V$ of $K$-varieties.
The collection of \'etale images in $V(K)$ forms a basis for the \textbf{\'etale open topology} (or \textbf{$\Sa E_K$-topology}) on $V(K)$.
Fact~\ref{fact:basic} below gathers some basic facts on the \'etale open topology from \cite{firstpaper}.  Parts (\ref{part1}) and (\ref{part2}) are \cite[Theorem A, Lemma 5.3]{firstpaper}, while parts (\ref{part3}) and (\ref{part4}) are immediate consequences of (\ref{part1})--(\ref{part2}).

\begin{fact}
\label{fact:basic}
Suppose that $V \to W$ is a morphism of $K$-varieties.
Then:
\begin{enumerate}
\item \label{part1} the induced map $V(K) \to W(K)$ is $\Sa E_K$-continuous.
\item \label{part2} if $V \to W$ is \'etale then the induced map $V(K) \to W(K)$ is $\Sa E_K$-open.
\item \label{part3} the map $K \to K$, $x \mapsto \alpha x + \beta$ is an $\Sa E_K$-homeomorphism for any $\alpha \in K^\times, \beta \in K$.
\item \label{part4} if $n$ is prime to $\Chara(K)$ then $\{ \alpha^n : \alpha \in K^\times \}$ is an \'etale open subset of $K$.
\end{enumerate}
\end{fact}

We will need Fact~\ref{fact:refine} below, proven in \cite[Lemma 4.2, Lemma 4.8]{firstpaper}.

\begin{fact}
\label{fact:refine}
Let $\Sa T$ be a Hausdorff field topology on $K$.
If the $\Sa T$-topology on each $K^n = \Aa^n(K)$ refines the $\Sa E_K$-topology, then the $\Sa T$-topology on $V(K)$ refines the $\Sa E_K$-topology for any $K$-variety $V$.
If the $\Sa E_K$-topology on $K$ refines $\Sa T$, then the $\Sa E_K$-topology on $V(K)$ refines the $\Sa T$-topology for any $K$-variety $V$.
\end{fact}

We finally prove a general lemma.

\begin{lemma}
\label{lem:etale-adic}
Suppose $R$ is a local domain and $K = \Frac(R)$.
The following are equivalent:
\begin{enumerate}
\item The $\Sa E_K$-topology on $V(K)$ refines the $R$-adic topology for any $K$-variety $V$.
\item The $\Sa E_K$-topology on $K$ refines the $R$-adic topology.
\item $R$ is an $\Sa E_K$-open subset of $K$.
\item $R$ contains a nonempty $\Sa E_K$-open subset of $K$.
\end{enumerate} 
If $K$ is additionally large and perfect then (1)-(4) is equivalent to the following:
\begin{enumerate}
\setcounter{enumi}{4}
\item $R$ contains an infinite set which is existentially definable in $K$.
\item There is a morphism $f \colon V \to \Aa^1$ of $K$-varieties such that $f(V(K))$ is infinite and contained in $R$.
\end{enumerate}
\end{lemma}

We will only use the equivalence of (1)-(4) at present.
If $K$ is not large then the \'etale open topology is discrete and hence trivially refines the $R$-adic topology.

\begin{proof}
Fact~\ref{fact:refine} shows that (1)$\Leftrightarrow$(2).
It is clear that (2)$\Rightarrow$(3).
The definition of the $R$-adic topology and Fact~\ref{fact:basic}.3 together show that (3)$\Rightarrow$(2).
The equivalence of (3) and (4) follows by Fact~\ref{fact:basic}.3 as $R$ is an additive subgroup of $K$.
Suppose that $K$ is large and perfect.
By \cite[Theorem C]{firstpaper} the $\Sa E_K$-topology on $K$ is not discrete, so any nonempty \'etale image in $K$ is infinite.
If (4) holds then $R$ contains a nonempty \'etale image, (6) follows immediately and (5) follows as \'etale images are existentially definable.
(6) implies (5) as $f(V(K))$ is existentially definable.
Finally, an infinite existentially definable subset of $K$ has nonempty $\Sa E_K$-interior~\cite[Theorem B]{secondpaper}, hence (5) implies (4).
\end{proof}

\section{\texorpdfstring{$R$}{R}-adic topologies and V-topologies}
\label{section:adic topologies and V topologies}
Let $\Sa T$ be a Hausdorff non-discrete field topology on $K$.
Then $\Sa T$ is a \textbf{V-topology} if for every neighbourhood $U \ni 0$ there is a neighborhood $V \ni 0$ such that $xy \in V \implies x \in U ~ \vee ~ y \in U$.  A field topology $\Sa T$ is a V-topology if and only if $\Sa T$ is induced by a non-trivial valuation or absolute value~\cite[Theorem~B.1]{EP-value}.

\begin{proposition}
\label{prop:V-topology}
Suppose that $R$ is an $n$-dimensional Noetherian local domain with fraction field $K$.
If $n \ge 2$ then the $R$-adic topology on $K$ is not a V-topology.
Hence if $R$ is in addition regular then the $R$-adic topology is a V-topology if and only if $n = 1$.
In particular, the $L[[t_1,\ldots,t_n]]$-adic topology on $L((t_1,\ldots,t_n))$ is a V-topology if and only if $n = 1$.
\end{proposition}

A Noetherian local ring has finite Krull dimension \cite[8.2.2]{Eisenbud}, so our use of ``$n$''
 is justified. The proof of Proposition~\ref{prop:V-topology} requires the following lemmas.

\begin{lemma}
\label{lem:2-val}
Suppose that $R$ is a local domain, $K$ is the fraction field of $R$, $v,v^* \colon K^\times \to \Zz$ are homomorphisms such that $v,v^*$ are both non-negative on $R \setminus \{0\}$, and some $\beta \in K^\times$ satisfies $v(\beta) < 0 < v^*(\beta)$.
Then the $R$-adic topology on $K$ is not a V-topology.
\end{lemma}

Before proving Lemma~\ref{lem:2-val}, we use it to show that the $L[[t,t^*]]$-adic topology on $L((t,t^*))$ is not a V-topology.  
Let $v$ be the $t$-adic valuation and $v^*$ be the $t^*$-adic valuation on $L((t,t^*))$.  
Set $\beta = t^*/t$.
Then $v,v^*$ are both non-negative on $L[[t,t^*]]$ and $v(\beta) < 0 < v^*(\beta)$.

\begin{proof}[Proof (of Lemma~\ref{lem:2-val})]
Suppose otherwise.
Then there is $c \in K^\times$ such that $xy \in cR \implies x \in R ~ \vee ~ y \in R$.

For every $n \ge 1$ we have $v(\beta^n) < 0$, and if $n$ is sufficiently large then $v^*(c/\beta^n) = v^*(c) - n v^*(\beta) < 0$.
Thus there is $n$ such that $\beta^n,c/\beta^n \notin R$,
but $(\beta^n)(c/\beta^n) \in cR$,
a contradiction.
\end{proof}

\begin{fact}[{\cite[Theorem 144]{kaprings}}]\label{fact:2-prime}
Suppose $R$ is a Noetherian domain with dimension $\geq 2$.
Then $R$ contains infinitely many distinct prime ideals of height one. 
\end{fact}

We now prove Proposition~\ref{prop:V-topology}.
We let $R_\pfrak$ be the localization of a ring $R$ at a prime ideal $\pfrak$.

\begin{proof}
The second claim follows from the first as a one-dimensional regular local ring is a DVR.
We prove the first claim.
Applying Fact~\ref{fact:2-prime} fix height one prime ideals $\pfrak\ne\pfrak^*$ in $R$.
Then $R_\pfrak, R_{\pfrak^*}$ are one-dimensional Noetherian domains.
Let $v \colon R \setminus \{0\} \to \Zz$ be given by declaring $v(a)$ to be the length of the $R_\pfrak$-module $R_\pfrak/a R_\pfrak$.
Then $v$ extends to a map $K^\times \to \Zz$ by declaring $v(a/b) = v(a) - v(b)$.
We likewise define $v^* \colon K^\times \to \Zz$ using $\pfrak^*$.
By \cite[Definition~A.3]{Fulton-intersection} $v,v^*$ are well-defined homomorphisms.
Note that $v,v^*$ are non-negative on $R$ as the length is non-negative.
As $\pfrak,\pfrak^*$ are distinct height one prime ideals neither is contained in the other.
Fix $\alpha \in \pfrak \setminus \pfrak^*$, $\alpha^* \in \pfrak^* \setminus \pfrak$, and set $\beta = \alpha^*/\alpha$.
Note that $v(\beta) < 0 < v^*(\beta)$.
Apply Lemma~\ref{lem:2-val}.
\end{proof}

The valuational approach to $L((t,t^*))$ generalizes to the regular case.
Suppose that $R$ is a regular local ring of dimension $\ge 2$.
Let $\pfrak,\pfrak^*$ be distinct height one prime ideals of $R$.
A localization of a regular ring is regular~\cite[Corollary 19.14]{Eisenbud} so $R_\pfrak, R_{\pfrak^*}$ are one-dimensional regular local rings, hence DVR's.
Note that the induced discrete valuations on $K$ satisfy Lemma~\ref{lem:2-val}.

\section{Proof of Theorem~1.2.1}
\label{section:proof-1}
\textbf{In this section we suppose that $R$ is a Henselian local domain with maximal ideal $\mfrak$ and fraction field $K$.}
We will need the following variant of Hensel's lemma.

\begin{fact}
\label{fact:hensel-ish}
Suppose that $g \in R[x]$ and $\alpha \in R$ satisfy $g'(\alpha) \ne 0$ and $g(\alpha) \equiv 0 \pmod{g'(\alpha)^2 \mfrak}$.
Then there is $\alpha^* \in R$ such that $g(\alpha^*) = 0$ and $\alpha^* - \alpha \in (g(\alpha)/g'(\alpha))R$.
\end{fact}

We let $\overline{a} \in R/\mfrak$, $\overline{p}(y) \in (R/\mfrak)[y]$ be the reduction mod $\mfrak$ of an $a \in R$, $p \in R[y]$, respectively.

\begin{proof}
 After possibly replacing $g(x)$ with $g(x + \alpha)$ we suppose $\alpha = 0$.
Fix $c_0,\ldots,c_n \in R$ with $g(x) = c_0 + c_1 x + \ldots + c_n x^n$.
Then $c_1 = g'(0) \ne 0$ and $c_0/c^2_1 = g(\alpha)/g'(\alpha)^2 \in \mfrak$.
Let $p \in R[x]$ be given by
\begin{equation*}
p(x) = x + \sum_{i = 2}^n c_i c_1^{i-2} \left(\frac{c_0}{c^2_1}\right)^{i - 1}\! (x-1)^i.
\end{equation*}
As $c_0/c^2_1 \in \mfrak$ we have $\overline{p}(x) = x$.
Hence $\overline{p}(0) = 0$ and $\overline{p'}(0) = 1$.
As $R$ is Henselian there is $\beta \in \mfrak$ such that $p(\beta) = 0$.
Then
\[
0 = p(\beta) = \beta + \sum_{i = 2}^{n} c_i c^{i - 2}_1 \left(\frac{c_0}{c^2_1}\right)^{i - 1}\! (\beta - 1)^i 
= 1 + (\beta - 1) + \sum_{i = 2}^{n} c_i\left( \frac{c^{i - 1}_0}{c^i_1}\right)(\beta - 1)^i.
\]
Hence
\begin{align*}
0 = c_0 p(\beta) &= c_0 + c_0(\beta - 1) + \sum_{i = 2}^{n} c_i\left( \frac{c_0}{c_1}\right)^i (\beta - 1)^i \\
&= c_0 + c_1 \left(\frac{c_0}{c_1}\right) (\beta - 1) + \sum_{i = 2}^{n} c_i \left( \frac{c_0}{c_1}\right)^i (\beta - 1)^i \\
&= \sum_{i = 0}^{n} c_i \left(\frac{c_0}{c_1}\right)^i (\beta - 1)^i
\end{align*}
Set $\alpha^* = (c_0/c_1)(\beta - 1)$.
Then $\alpha^* \in (c_0/c_1)R$ as $\beta \in R$.
Note that $g(\alpha^*) = 0$.
\end{proof}

We proceed with the proof of Theorem~\ref{thm:main}.1.
We recall the notion of a standard \'etale morphism; see \cite[Definition 3.5.38]{poonen-qpoints} for more details. A morphism $f:X\to Y$ between affine schemes is called \textbf{standard \'etale} if the corresponding map in the category of rings is of the form, $R\to R[x]_g/(f)$ where $g,f\in R[x]$ and $f$ is monic with $f'$ invertible in $R[x]_g/(f)$.
Given an affine $K$-variety $V$, a \textbf{standard \'etale image} in $V(K)$ is a set of the form $f(X(K))$ for a standard \'etale morphism $f \colon X \to V$ of $K$-varieties.  If $V$ is affine, then every \'etale image is a union of standard \'etale images.  This holds because for any \'etale morphism $f \colon X \to V$, we can cover $X$ by affine opens $U \subseteq X$ such that $U \to V$ is standard \'etale \cite[Lemma 02GT]{stacks-project}.

We show that the $R$-adic topology on $V(K)$ refines the $\Sa E_K$-topology for any $K$-variety $V$.
By Fact~\ref{fact:refine} it suffices to fix $n$ and show that the $R$-adic topology on $K^n$ refines the $\Sa E_K$-topology.
It suffices to fix an $\Sa E_K$-open neighbourhood $O \subseteq K^n$ of the origin and show that $0$ lies in the $R$-adic interior of $O$.
We may suppose that $O$ is a standard \'etale image.
By the definition of a standard \'etale morphism there are $f, g \in K[x_1,\ldots,x_n,y]$ such that:
\begin{enumerate}
\item $O = \{ \alpha \in K^n : (\exists \beta \in K) f(\alpha,\beta) = 0 \ne g(\alpha,\beta) \} $, and
\item $\frac{\partial f}{\partial y} (\alpha,\beta) \ne 0$ for all $(\alpha,\beta) \in K^n \times K$ such that $f(\alpha,\beta) = 0 \ne g(\alpha,\beta)$.
\end{enumerate}
Fix $b \in K$ such that $f(0,b) = 0 \ne g(0,b)$.
Replacing $y$ with $y+ b$, we may suppose $b = 0$, so that $f(0,0) = 0 \ne g(0,0)$.
After clearing denominators we may suppose $f \in R[x_1,\ldots,x_n,y]$.
This implies that $\frac{\partial f}{\partial y} \in R[x_1,\ldots,x_n]$, hence $\frac{\partial f}{\partial y}(0,0) \in R$.
Because $g(x,y)$ is $R$-adically continuous and $g(0,0) \ne 0$, there is some $R$-adic neighbourhood $U$ of 0 in $K$ such that $g(x,y) \ne 0$ whenever $(x,y) \in U^n \times U$.
Because $R$ is $R$-adically bounded, there is an $R$-adic neighbourhood $U^*$ of 0 such that $U^* \cdot R \subseteq U$.
For each $a \in K^n$ we let $f_a \in K[y]$ be given by $f_a(y) = f(a,y)$.
Note that $f'_a(y) = \frac{\partial  f}{\partial y} (a,y)$.
Thus $f_0 \in R[y]$, $f_0(0) = 0$, and $f'_0(0) \ne 0$.
By continuity there is an $R$-adic neighbourhood $U'$ of $0$ in $K^n$ such that $U' \subseteq U^n$ and if $a \in U'$ then:
\begin{enumerate}
\item $f_a \in R[y]$,
\item $f'_a(0) \ne 0$,
\item $f_a(0)/f'_a(0)^2 \in \mfrak$,
\item $f_a(0)/f'_a(0) \in U^*$.
\end{enumerate}
We show that $U' \subseteq O$.
Fix $\alpha \in U'$.
Then we have $f'_\alpha(0) \ne 0$ and $f_\alpha(0) \equiv 0 \pmod{f'_\alpha(0)^2\mfrak}$.
By Fact~\ref{fact:hensel-ish} there is $\beta \in (f_\alpha(0)/f'_\alpha(0))R$ such that $f_\alpha(\beta) = 0$.
By (4) we have $\beta \in U^* \cdot R$.
Hence $\beta \in U$, so $(\alpha,\beta) \in U^n \times U$, and $g(\alpha,\beta) \ne 0$.
Note $f(\alpha,\beta) = 0$, so $\beta$ witnesses $\alpha \in O$.

\section{Proof of Theorem~1.2.2}
\label{section:proof-3}
\textbf{In this section we suppose that $R$ is a regular local ring with fraction field $K$.}
We first prove several algebraic lemmas.
Lemma~\ref{lem:regular-lem} is presumably unoriginal.

\begin{lemma}
\label{lem:regular-lem}
Suppose that $S$ is the Henselization of $R$.
Then $S \cap K = R$.
\end{lemma}

\begin{proof}
The Henselization $S$ is faithfully flat over $R$~\cite[Lemma~07QM]{stacks-project}.  For $a,b\in R$ with $a/b\in S$ (i.e. $a\in bS$), $bS\cap R=bR$ by faithful flatness~\cite[Chapter 2, (4.C)(ii)]{matsumura}. Hence $a\in bR$ and $a/b\in R$.
\end{proof}

\begin{lemma}
\label{lem:JL}
Suppose that $\dim R \ge 2$.
\begin{enumerate}
\item If $\Chara(R/\mfrak) \ne 2$ and $\alpha,\beta \in K$ satisfy $1 + \alpha^4 = \beta^2$ then $\alpha \in R$ or $1/\alpha \in R$.
\item If $\Chara(R/\mfrak) = 2$ and $\alpha,\beta \in K$ satisfy $1 + \alpha^3 = \beta^3$ then $\alpha \in R$ or $1/\alpha \in R$.
\end{enumerate}
\end{lemma}

\begin{proof}
We reduce to the special case where $R$ is Henselian, which is proven in \cite[pg 52, 55]{model-theoretic-algebra}.
We treat (1); (2) follows in the same manner.
Suppose that $\Chara(R/\mfrak) \ne 2$ and $\alpha,\beta \in K$ satisfy $1 + \alpha^4 = \beta^2$.
Let $S$ be the Henselization of $R$.
Then $S$ is a regular local ring of the same dimension as $R$ by~\cite[Lemmas~06LJ, 06LK, 06LN]{stacks-project}.
The maximal ideal of $S$ is $\mfrak S$ and $S/\mfrak S = R/\mfrak$~\cite[Lemma~07QM]{stacks-project}.
Hence $\Chara(S/\mfrak S) \ne 2$.
By the Henselian case,
$\alpha \in S$ or $1/\alpha \in S$.
An application of Lemma~\ref{lem:regular-lem} shows that $\alpha \in R$ or $1/\alpha \in R$.
\end{proof}

We proceed with the proof of Theorem~\ref{thm:main}.2.
If $R$ is one dimensional then $K$ is a discretely valued field; this case follows by Fact~\ref{fact:canonical}.5.
Hence we suppose that $\dim R\ge 2$.
By Lemma~\ref{lem:etale-adic} it is enough to produce an $\Sa E_K$-neighbourhood of $0$ contained in $R$.
First suppose $\Chara(R/\mfrak) \ne 2$.
This implies that $\Chara(K) \ne 2$.
Let $S$ be $\{ b^2 : b \in K^\times \}$, $f \colon K \to K$ be given by $f(x) = 1 + x^4$, and $\Omega = f^{-1}(S)$.
Note $0 \in \Omega$.
By Fact~\ref{fact:basic}, $S$ and $\Omega$ are $\Sa E_K$-open sets.
By Lemma~\ref{lem:JL}, $\Omega \subseteq R \cup R^{-1}$.
Fix a height one prime ideal $\pfrak$ of $R$.
The localization $R_\pfrak$ is a regular local ring of dimension~1, hence a discrete valuation ring.
Let $v \colon K^\times \to \Zz$ be the associated valuation.
By Fact~\ref{fact:canonical}.5, the $\Sa E_K$-topology on $K$ refines the $v$-topology. 
Therefore the valuation ideal $\pfrak R_\pfrak = \{a \in K \colon v(a) > 0\}$ is $\Sa E_K$-open.
The intersection $\Omega \cap \pfrak R_\pfrak$ is an $\Sa E_K$-open neighbourhood of 0.
We show that $\Omega \cap \pfrak R_\pfrak \subseteq R$. Suppose $a \in \Omega\cap \pfrak R_\pfrak$.
As $a \in \pfrak R_\pfrak$ we have $v(a) > 0$, hence $v(a^{-1}) < 0$, hence $a^{-1} \notin R_\pfrak$, so $a^{-1} \notin R$.
As $a \in \Omega$, we have $a \in R$.

The case when $\Chara(R/\mfrak) = 2$ is similar, using $S = \{b^3 : b \in K^\times\}$ and $f(a) = 1 + a^3$.

\section{Proof of Theorem~1.3}
Fix a prime $p$, let $\res \colon \Zz_p \to \Zz/p\Zz$ be the residue map, and $v \colon \Qq^\times_p \to \Zz$ be the $p$-adic valuation.
Fix a $\Qq$-linear derivation $\der\colon\Qq_p \to \Qq_p$ that is not identically zero. 
We can construct $\der$ in the following fashion: Fix $t \in \Qq_p$ transcendental over $\Qq$, let $\der^*$ be the unique $\Qq$-linear derivation $\Qq(t) \to \Qq(t)$ such that $\der^* t = 1$, and $\der$ be a $\Qq$-linear derivation $\Qq_p \to \Qq_p$ extending $\der^*$.
Such extensions exist by \cite[Chapter~1 \textsection 5]{weil-fga}.
By $\Qq$-linearity $\der r=0$ for all $r\in \Qq$.

We define $E$ to be $\{a \in \Zz_p : \der a \in \Zz_p\}$.
It is easy to see that $E$ is a subring of $\Zz_p$.
In this section we show that $E$ is a one-dimensional Noetherian Henselian local ring, $\Qq_p$ is the fraction field of $E$, and the $E$-adic topology on $\Qq_p$ strictly refines the $p$-adic topology.
We also show that $E$ has non-reduced completion, so $E$ is neither regular nor excellent.

\begin{lemma}
\label{lem:density}
The graph $\{(a,\der a) : a \in \Qq_p\}$ of $\der$ is $p$-adically dense in $\Qq_p^2$.
\end{lemma}

\begin{proof}
Fix $t \in \Qq_p$ with $\der t \ne 0$.
Recall that $\Qq^2$ is dense in $\Qq_p^2$.
Let $f \colon \Qq^2_p \to \Qq^2_p$ be the affine transformation $f(x,y) = (x + y t, (\der t) y)$.
Then $f$ is a homeomorphism as $\der t \ne 0$, hence $f(\Qq^2)$ is $p$-adically dense in $\Qq^2_p$.
We have $\der(a + bt) = \der(a) + b \der(t) = b\der(t)$ for all $a,b \in \Qq$, hence
\begin{equation*}
f(\Qq^2) =
 \{(a+bt, \der(a + bt)) : a, b \in \Qq\} 
    \subseteq \{(c,\der c) : c \in \Qq_p\}. \qedhere
\end{equation*}
\end{proof}

\begin{lemma}
\label{lem:fraction-field}
The fraction field of $E$ is $\Qq_p$.
\end{lemma}

\begin{proof}
It is enough to show that $\Zz_p$ is contained in the fraction field of $E$.
Fix $a \in \Zz_p$.
We produce $b \in \Zz_p$ such that $b,ab \in E$.
Take $\gamma \in \Nn$ such that $\gamma + v(\der a) \ge 0$.
By Lemma~\ref{lem:density} there is $b \in \Qq_p$ satisfying $v(b) = v(\der b) = \gamma$; in particular $b \ne 0$.
Then $b, \der b \in \Zz_p$ as $\gamma \ge 0$, hence $b \in E$.
As $a,b \in \Zz_p$ we have $ab \in \Zz_p$, so it remains to show that $\der(ab) \in \Zz_p$.
We have 
\begin{align*}
v(\der(b)a) \hspace{1.5pt}&= \hspace{1.5pt} v(\der b) + v(a) \hspace{1.5 pt} = \hspace{1.5pt} \gamma + v(a) \hspace{1.5pt}\ge \hspace{1.5pt} 0 \\
v(\der(a)b) &= v(\der a) + v(b) = v(\der a) + \gamma \ge 0.
\end{align*}
Hence $a\der(b), b\der(a) \in \Zz_p$, so $\der(ab) = a\der(b) + b\der(a)$ is in $\Zz_p$.
\end{proof}

\begin{lemma}
\label{lem:max-ideal}
$E$ is a local ring with maximal ideal $E \cap p\Zz_p = \{a \in E : \res(a) = 0 \}$.
\end{lemma}

\begin{proof}
Let $\mfrak = E \cap p\Zz_p$.
Then $\mfrak$ is a proper ideal of $E$.
It suffices to fix $a \in E \setminus \mfrak$ and show that $a$ is invertible in $E$.
We have $a \in \Zz_p$ and $\res(a) \ne 0$, hence $v(a) = v(a^{-1}) = 0$.
Then
\[
v(\der(a^{-1})) = v( - \der(a)/a^2) = v(\der a) - 2v(a) = v(\der a) \ge 0.
\]
Therefore $a^{-1}, \der(a^{-1}) \in \Zz_p$, hence $a^{-1} \in E$.
\end{proof}

\begin{lemma}
\label{lem:adic-refinement}
Suppose that $R,R^*$ are local domains with the same fraction field $K$.
Then the $R$-adic topology on $K$ refines the $R^*$-adic topology if and only if $R^*$ is $R$-adically open.
If $R \subseteq R^*$ then the $R$-adic topology refines the $R^*$-adic topology.
\end{lemma}

\begin{proof}
The first claim is easy and left to the reader.
Suppose $R \subseteq R^*$.
Then $R$ is an additive subgroup of $R^*$, hence $R^* = \bigcup_{a \in R^*} (a + R)$, hence $R^*$ is $R$-adically open.
\end{proof}

\begin{proposition}
\label{prop:refine}
The $E$-adic topology on $\Qq_p$ strictly refines the $p$-adic topology.
\end{proposition}

\begin{proof}
We have $E \subseteq \Zz_p$, so the $E$-adic topology refines the $p$-adic topology by Lemma~\ref{lem:adic-refinement}.
We show that $E$ is not open in the $p$-adic topology.
Suppose $O$ is a nonempty $p$-adically open subset of $\Qq_p$.
By Lemma~\ref{lem:density} there is $a \in O$ such that $\der a \in \Qq_p \setminus \Zz_p$.
Then $a \notin E$, hence $O \not\subseteq E$.
\end{proof}

\begin{proposition}
\label{prop:example-hensel}
$E$ is Henselian.
\end{proposition}

\begin{proof}
Given $g \in \Qq_p[x], g(x) = a_0 + a_1 x + \ldots + a_d x^d$ we let $\der g \in \Qq_p[x]$ be $\der(a_0) + \der(a_1) x + \ldots + \der(a_d) x^d$.
Note that if $g \in E[x]$ then $\der g \in \Zz_p[x]$.
As above we let $\mfrak$ be the maximal ideal of $E$.

Fix $g \in E[x]$ and $a \in E$ such that $g(a) \equiv 0 \pmod{\mfrak}$ and $g'(a) \not\equiv 0 \pmod{\mfrak}$.
We will produce $a^* \in E$ such that $g(a^*) = 0$ and $a^* \equiv a \pmod{\mfrak}$.
As $\mfrak = E \cap  p\Zz_p$ we have $g(a) \equiv 0 \pmod{p\Zz_p}$ and $g'(a) \not\equiv 0 \pmod{p\Zz_p}$.
As $\Zz_p$ is Henselian there is $a^* \in \Zz_p$ such that $g(a^*) = 0$ and $a^* \equiv a \pmod{p\Zz_p}$.
It suffices to show that $a^* \in E$: as $\mfrak = E \cap p \Zz_p$ this will yield $a^* \equiv a \pmod{\mfrak}$, and hence that $E$ is Henselian. To show that $a^* \in E$, it is enough to show that $\der a^* \in \Zz_p$.
We have
\[
0 = \der(0) = \der(g(a^*)) = (\der g)(a^*) + g'(a^*) \der(a^*)
\]
hence
\[
\der(a^*) = \frac{(\der g)(a^*)}{-g'(a^*)}.
\]
As $g \in E[x]$, we have $\der g \in \Zz_p[x]$, hence $(\der g)(a^*) \in \Zz_p$.
As $a^* \equiv a \pmod{p\Zz_p}$ we have $g'(a^*) \equiv g'(a) \not \equiv 0 \pmod{p\Zz_p}$, and so $g'(a^*) \in \Zz_p^\times$. Therefore $\der a^* \in \Zz_p$.
\end{proof}

It remains to show that $E$ is one-dimensional Noetherian.  We will use the following fact.

\begin{fact}
\label{fact:cohen}
Suppose that $R$ is a domain, $R$ is not a field, and there is $n$ such that every ideal in $R$ admits an $n$ element generating set.
Then $R$ is one-dimensional Noetherian.
\end{fact}
Fact~\ref{fact:cohen} is a theorem of Cohen~\cite[Corollary 1, Theorem 9 and Theorem 10]{cohen-rm}.  To apply Fact~\ref{fact:cohen}, we will show that every ideal is generated by two elements.  We first prove a technical lemma.

\begin{lemma} \label{3-to-2}
Suppose $a, a', a'' \in \Qq_p$, $v(a) \le \min\{v(a'),v(a'')\}$ and $v(\der(a'/a)) \le v(\der(a''/a))$.  
Then $a'' \in aE + a'E$.
\end{lemma}

\begin{proof}
We may assume $a'' \ne 0$.
Then $v(a) \le v(a'') < \infty$, so $a \ne 0$.
After replacing $a, a', a''$ by $a/a, a'/a, a''/a$, we may suppose $a = 1$.
Then $0 = v(a) \le \min\{v(a'),v(a'')\}$, so $a', a'' \in \Zz_p$.  Additionally, $v(\der a') \le v(\der a'')$.
This yields $b \in \Zz_p$ such that $\der (a'') = b \der(a')$.
By continuity of multiplication there is a $p$-adically open neighbourhood $U \subseteq \Zz_p$ of $b$ such that if $b^* \in U$ then $\der(a'') - b^* \der(a') \in \Zz_p$.
By Lemma~\ref{lem:density} there is $b^* \in U$ such that $\der(b^*) \in \Zz_p$.
Then $b^* \in E$. 
Let $c = a'' - b^*a'$.
Then $a'' = c + b^* a'$.
We claim that $c \in E$, so that $a'' \in E + a' E$.
We have $c \in \Zz_p$, as $a',a'',b^* \in \Zz_p$.
We have
\[
\der(c) = \der(a'') - \der(b^*a') = \der(a'') - b^* \der(a') - a'\der(b^*).
\]
Note that $\der(a'') - b^*\der(a')$ and $a'\der(b^*)$ are both in $\Zz_p$.
Hence $\der c \in \Zz_p$, so $c \in E$.
\end{proof}

\begin{proposition}
$E$ is one-dimensional Noetherian.
\end{proposition}

\begin{proof}
By Fact~\ref{fact:cohen} it suffices to show that any ideal $I$ in $E$ has a two element generating set.
We may suppose $I \ne \{0\}$.
As $E \subseteq \Zz_p$ we have $v(a) \ge 0$ for all $a \in I$.
Fix $a \in I$ minimizing $v(a)$; note that $a \ne 0$.
For any $a^* \in I$ we have
\[
v(\der(a^*/a)) = v \left( \frac{a\der(a^*) - a^*\der(a)}{a^2} \right) = v(a\der(a^*) - a^*\der(a)) - 2v(a).
\]
As $a,a^*,\der (a), \der(a^*) \in \Zz_p$ we have $v(a\der(a^*) - a^*\der(a)) \ge 0$, hence $v(\der(a^*/a)) \ge -2v(a)$.
Therefore we may select $a' \in I$ minimizing $v(\der(a'/a))$.
We show that $I = aE + a'E$.
Fix $a'' \in I$.
Then $v(a) \le \min \{ v(a'), v(a'')\}$ and $v(\der(a'/a)) \le v(\der(a''/a))$.
Apply Lemma~\ref{3-to-2}.
\end{proof}

\begin{remark}
\label{remark}
Our ring $E$ is similar to Ferrand and Raynaud's example of a Noetherian 1-dimensional local domain with non-reduced completion~\cite[Proposition~3.1]{ferrand-raynaud}.
Their example probably satisfies an analogue of Theorem~\ref{thm:example}, with $\Cc\{t\}$ replacing $\Qq_p$.
Likewise, our $E$ also has non-reduced completion, by Lemma~\ref{lem:completion} below.
This implies that $E$ is not excellent \cite[07QT, 07GH, 07QK]{stacks-project} and not regular \cite[07NY, 00NP]{stacks-project}.
\end{remark}
\begin{lemma} \label{lem:completion}
The completion of $E$ is $\Zz_p[x]/(x^2)$.
\end{lemma}
\begin{proof}
It is enough to produce a ring embedding $\uptau \colon E \to \Zz_p[x]/(x^2)$ such that $\uptau$ gives a dense topological embedding from the $\mfrak$-adic topology on $E$ to the $p$-adic topology on $\Zz_p[x]/(x^2)$.
Let $\uptau \colon E \to \Zz_p[x]/(x^2)$ be $\uptau(a) = a + \der(a) x$.
Note that $\uptau$ is an injective ring homomorphism.
The $p$-adic topology on $\Zz_p[x]/(x^2)$ agrees with the product topology given by the natural bijection $\Zz_p[x]/(x^2) \to \Zz_p^2$.  
By Lemma~\ref{lem:density} the image of $\uptau$ is dense.
The $\mfrak$-adic topology on $E$ agrees with the restriction of the $E$-adic topology on $\Qq_p$ to $E$.
(It suffices to show that for any non-zero ideal $I$ in $E$ we have $\mfrak^n \subseteq I$ for some $n$.
This holds because $E/I$ is a local Artinian ring, as $\dim E = 1$.)
Meanwhile, $(\Qq_p,\Zz_p,\partial)$ is a ``lifted diffeovalued field'' (see \cite[Definition 8.10]{dp-finite-iv}\footnote{In~\cite[Section 8]{dp-finite-iv}, the residue characteristic $0$ assumption was only used in~\cite[Proposition 8.9]{dp-finite-iv}. Hence the definitions and results quoted in the proof are unaffected.}) that is ``dense'' (see \cite[Definition 8.12]{dp-finite-iv}) by~\cite[Proposition 8.19]{dp-finite-iv} and Lemma~\ref{lem:density} above.  The $E$-adic topology on $\Qq_p$ is the ``diffeovaluation topology'' (see~\cite[Definition 8.16]{dp-finite-iv}). Therefore~\cite[Proposition~8.21]{dp-finite-iv} shows that the collection of sets of the form $\{ \alpha \in E : v(\beta - \alpha) > \gamma \text{  and  } v(\beta^* - \der \alpha) > \gamma^* \}$ for $\beta,\beta^* \in E$ and $\gamma, \gamma^* \in \Zz$ is a basis for the $E$-adic topology on $\Qq_p$.  This implies that $\uptau : E \to \Zz_p^2$ is a topological embedding.

\end{proof}

\begin{remark}
We discuss our usage of the axiom of choice.
Note that $\der$ is a discontinuous additive homomorphism $\Qq_p \to \Qq_p$, because $\der$ is non-zero but vanishes on the dense set $\Qq \subseteq \Qq_p$.
This implies that $\der$ is not measurable, hence existence of $\der$ requires a strong application of the axiom of choice; see for example \cite[Section~2]{automatic-continuity}.
One can avoid this.
In fact, our argument goes through for $R$ a characteristic zero Henselian DVR, $K = \Frac(R)$, and $\der \colon K \to K$ a derivation with dense graph.
For example fix $t \in \Qq_p$ transcendental over $\Qq$ and let $F$ be the algebraic closure of $\Qq(t)$ in $\Qq_p$.
Then $F$ is a dense subfield of $\Qq_p$ and $F \cap \Zz_p$ is a Henselian DVR with fraction field $F$.
Let $\der^*$ be the unique $\Qq$-linear derivation $\Qq(t) \to \Qq(t)$ with $\der^* t = 1$.
Then $\der^*$ uniquely extends to a $\Qq$-linear derivation $\der\colon F \to F$ as $F/\Qq(t)$ is algebraic~\cite[Proposition~1.15]{weil-fga}.
This extension does not require choice.
Our example goes through with $\Zz_p,\Qq_p$ replaced by $F \cap \Zz_p, F$, respectively.  At the end, one obtains a $p$-adically closed field $F$ of transcendence degree 1 and a subring $E \subseteq F$ satisfying the conclusions of Theorem~\ref{thm:example}, with $\Qq_p$ replaced by $F$.\end{remark}

\section{Final remark}
\label{section:final}
We showed in \cite[Theorem B]{firstpaper} that the \'etale open topology over $K$ is induced by a V-topology if and only if $K$ is infinite, t-Henselian, and not separably closed (t-Henselianity is a topological generalization of Henselianity, see \cite[Section 7]{Prestel1978}).
We have produced examples such as $L((t_1,\ldots,t_n))$ for $n > 1$ where the \'etale open topology is induced by a field topology that is not a V-topology.
We do not know a general criterion for when the \'etale open topology is induced by a field topology.

\bibliographystyle{amsalpha}
\bibliography{ref}

\end{document}